\documentclass[12pt]{amsart}
\usepackage{amscd,amssymb,graphics}
\usepackage{hyperref}

\usepackage{amsfonts}
\usepackage{amsmath}

\newtheorem{thm}{Theorem}[section]
\newtheorem{fact}[thm]{Fact}
\newtheorem{corol}[thm]{Corollary}
\newtheorem{lemma}[thm]{Lemma}
\newtheorem{prop}[thm]{Proposition}

\newtheorem{quest}[thm]{Question}
\newtheorem{defi}[thm]{Definition}

\theoremstyle{remark}
\newtheorem{remark}[thm]{Remark}
\newtheorem{example}[thm]{Example}

\newcommand{\ben}{\begin{enumerate}}
\newcommand{\een}{\end{enumerate}}
\newcommand{\bit}{\begin{itemize}}
\newcommand{\eit}{\end{itemize}}
\newcommand{ \hXm}{(\widehat{X},\widehat{\mu})}
\def\R {{\Bbb R}}

\def\N{{\Bbb N}}
\def\T{{\Bbb T}}

\def\Z {{\Bbb Z}}

\def\Homeo{{\mathrm{Homeo}}\,}

\def\K {\mathcal K}
\def\NA  {\mathcal{NA}}

\def\QED{\nobreak\quad\ifmmode\roman{Q.E.D.}\else{\rm Q.E.D.}\fi}

\def\a {\alpha}

\begin{document}

\title[]
{Notes on non-archimedean topological groups}

\author[]{Michael Megrelishvili}
\address{Department of Mathematics,
Bar-Ilan University, 52900 Ramat-Gan, Israel}
\email{megereli@math.biu.ac.il}
\urladdr{http://www.math.biu.ac.il/$^\sim$megereli}


\author[]{Menachem Shlossberg}
\address{Department of Mathematics,
Bar-Ilan University, 52900 Ramat-Gan, Israel}
\email{shlosbm@macs.biu.ac.il}
 \urladdr{http://www.math.biu.ac.il/$^\sim$shlosbm}

\date{June 06, 2011}


\dedicatory{Dedicated to Professor Dikran Dikranjan on his 60th birthday}

\begin{abstract} 
We show that the Heisenberg type group $H_X=(\Bbb{Z}_2 \oplus
V)\leftthreetimes V^{\ast}$, with the discrete Boolean group
$V:=C(X,\Z_2)$, canonically defined by any Stone space $X$, is
always minimal. That is, $H_X$ does not admit any strictly coarser
Hausdorff group topology. This leads us to the following result: for
every (locally compact) non-archimedean $G$ there exists a (resp.,
locally compact) non-archimedean  minimal group $M$ such that $G$ is
a group retract of $M.$
 For discrete groups $G$ the latter was proved by S. Dierolf and U. Schwanengel \cite{DIS79}.
We unify some old and new characterization results for
non-archimedean groups.
 We show also that any epimorphism into a non-archimedean group must be dense.
\end{abstract}
\maketitle

\setcounter{tocdepth}{1}

 \tableofcontents

\section{Introduction and preliminaries}


 A topological group is {\it
non-archimedean} if it has a local base at the identity consisting
of open subgroups.
This class of groups 
coincides with the class of topological subgroups of the
homeomorphism groups
$\Homeo(X)$, where
 $X$ runs over \emph{Stone spaces} (=compact zero-dimensional
 spaces)
 and $\Homeo(X)$ carries the usual compact open topology.
 Recall that by Stone's representation theorem,
  there is a duality between the category of Stone spaces and the category of Boolean
 algebras.
 The class $\mathcal{NA}$ of non-archimedean groups and their actions on ultra-metric
  and Stone spaces
 have many applications. For instance, in non-archimedean functional analysis, in descriptive set theory, computer science, etc.
 See, e.g., \cite{ro, bk, Les86, Lem03} and references therein.

In the present paper we provide some applications of generalized
Heisenberg groups, with emphasis on minimality properties, in
the theory of $\mathcal{NA}$ groups and actions on Stone
spaces.

Recall that a Hausdorff topological group $G$ is 
{\it minimal} (Stephenson \cite{STE71} and Do\"\i chinov \cite{DOI72}) if it does not admit a strictly
coarser Hausdorff group topology, or equivalently, if every
injective continuous group homomorphism $G\rightarrow P$ into a
Hausdorff topological group is a 
topological group embedding.

If otherwise is not stated
all topological groups and spaces in this paper are assumed to be Hausdorff.
We say that an additive topological 
group $(G,+)$ is a \emph{Boolean group} if $x+x=0$ for every $x \in G$.
As usual, a \emph{$G$-space} $X$ is a topological space $X$ with a
continuous
group action $\pi: G \times X \to X$ of a topological group $G$. We say that $X$ is a \emph{$G$-group}
if, in addition, $X$ is a topological group and all $g$-translations,
$\pi^g: X \to X, \  x \mapsto gx:=\pi(g,x)$, are automorphisms of $X$. For every $G$-group $X$ we denote by $X \leftthreetimes G$ the corresponding
topological semidirect product.


\vskip 0.4cm

 To every Stone space $X$ we associate a (locally compact 2-step nilpotent) Heisenberg type group
 $$H_X=(\Bbb{Z}_2 \oplus V) \leftthreetimes V^{\ast},$$ where
$V:=C(X,\Z_2)$ is a discrete Boolean group which can be identified
with the group of all clopen subsets of $X$ (symmetric difference is
the group operation). $V^{\ast}:=Hom(V,\Z_2)$ is the compact group
of all group homomorphisms into the two element cyclic group $\Z_2$.
$V^{\ast}$ acts on $\Bbb{Z}_2 \oplus V$ in the following way: every
$(f,(a,x))\in V^{\ast}\times (\Bbb{Z}_2 \oplus V)$ is mapped to
$(a+f(x),x)\in \Bbb{Z}_2 \oplus V.$ The group operation on $H_X$ is
defined as follows: for
$$u_{1}=(a_{1},x_{1},f_{1}), \ u_{2}=(a_{2},x_{2},f_{2}) \in H_X$$ we
define $$u_{1}u_{2}=(a_{1}
+a_{2}+f_{1}(x_{2}),x_{1}+x_{2},f_{1}+f_{2}).$$ 

In Section \ref{s:main} we study some properties of $H_X$ and show
in particular (Theorem \ref{t:1}) that
the (locally compact) Heisenberg group $H_X=(\Bbb{Z}_2\times
V)\leftthreetimes V^{\ast}$ is minimal and non-archimedean for every
Stone space $X$.


Every
Stone space $X$ is naturally embedded into $V^*:=Hom(V,\Z_2)$
by the natural map 
 $\delta: X \to V^*, \ x \mapsto \delta_x$ where
$\delta_x(f):=f(x)$.
Every $\delta_x$ can be treated as a 2-valued measure on $X$.
Identifying $X$ with $\delta(X) \subset V^*$ we
get a restricted evaluation map $V \times X \to \Z_2$ which in fact
is the evaluation map of the Stone duality.
Note that the role 
of $\delta: X \to V^*$ for a compact space $X$ is similar to the
role of the Gelfand map $X \to C(X)^*$, representing $X$ via the
point measures.

For every action of a group $G \subset \Homeo(X)$ on a Stone space $X$
we can deal with a $G$-space version
of the classical Stone duality. The map $\delta: X \to V^*$ is a
$G$-map of $G$-spaces. Every continuous group action of $G$ on a
Stone space $X$ is \emph{automorphizable}
in the sense of \cite{Me-F} meaning that $X$ is a $G$-subspace of
a $G$-group $K$. This contrasts the case of general compact spaces
(see \cite{Me-F}).
More generally, we study 
(Theorem \ref{t:AE}) also metric and uniform versions of
automorphizable actions. 

Furthermore, a deeper analysis shows (Theorem \ref{t:2}) that every
topological subgroup $G \subset \Homeo(X)$ induces a continuous
action of $G$ on $H_X$ by automorphisms such that the corresponding
semidirect product $H_X \leftthreetimes G$ is a minimal group.

We then conclude (Corollary \ref{c:main}) that
every (locally compact) non-archimedean group is a group retract of
a (resp., locally compact) minimal non-archimedean group.
It covers a result of Dierolf and Schwanengel \cite{DIS79} (see also
Example \ref{exa:cyc} below) which asserts that every discrete group
is a group retract of a locally compact non-archimedean minimal
group.

Section \ref{s:min} contains additional motivating results and
questions.
Several interesting
applications of generalized Heisenberg groups can be found in the
papers \cite{MEG95, MEG98, MEG04, DTY06, MEG08, DG09, DIM04, SHL07}.

Studying the properties of the Heisenberg group $H_X$, we get
a unified approach to several (mostly known) equivalent
characterizations of the class $\mathcal{NA}$ of non-archimedean
groups (Lemma \ref{l:firstpart} and Theorem \ref{t:condit}). In
particular, we show that the class of all topological subgroups of
$Aut(K),$  for compact abelian groups $K,$  is precisely
$\mathcal{NA}.$


  A morphism $f: M \to G$
 in a category $\mathcal{C}$ is an {\it epimorphism} if there are no
 different morphisms $g,h: G\to F$ in $\mathcal{C}$ such that $gf=hf.$
 In the category of Hausdorff topological
 groups a morphism with dense range is obviously an epimorphism.
 K.H. Hofmann asked in late 1960's whether the converse is true.
 This \emph{epimorphism problem} was answered by Uspenskij \cite{Us-epic} in the negative.
Nevertheless, in many natural cases indeed the epimorphism $M \to G$
must be dense. For example, Nummela \cite{Num} has shown it in the
case that the co-domain $G$ is either locally compact or
having the coinciding left and right uniformities.
Using a criterion of Pestov \cite{Pest-epic} and the uniform
automorphizability of certain actions by non-archimedean groups
(see Theorem \ref{t:AE}) we prove in Theorem \ref{epic=dense} that
 any epimorphism $f: M \to G$ into a non-archimedean group $G$ must be dense.

\vskip 0.5cm
\noindent \textbf{Acknowledgment:} We thank Dikranjan for many
valuable ideas and concrete suggestions. We are indebted also to the
referee for several improvements.

\section{Minimality and group representations}
\label{s:min}

 Clearly, every compact topological group is minimal.
 Trivial examples of nonminimal groups
 are: the group $\Z$ of all integers (or any discrete infinite abelian group)
and $\R$, 
the topological group of all reals.
By a fundamental theorem of Prodanov and Stoyanov \cite{PRS85} every abelian
minimal group is precompact. 
  For more information
about minimal groups see review papers of Dikranjan
\cite{DIK03} and  Comfort-Hofmann-Remus \cite{CHR}, a book of
Dikranjan-Prodanov-Stoyanov \cite{DPS89} and a recent book of
Lukacs \cite{Lukacs}.

Unexpectedly enough many non-compact naturally defined topological
groups
are minimal.

\begin{remark} \label{r:min}
Recall some nontrivial examples of minimal groups.
 \ben

\item
Prodanov \cite{PRO71} showed that the $p$-adic topologies are
the only precompact minimal group topologies on ${\Z}$. 
\item Symmetric topological groups $S_X$ (Gaughan \cite{GAU67}).
\item $\Homeo(\{0,1\}^{\aleph_0})$ (see Gamarnik \cite{Ga})
and also Uspenskij \cite{USP01} for a more general case).
\item $\Homeo[0,1]$ (Gamarnik \cite{Ga}).
\item
The semidirect product $\R \leftthreetimes \R_+$
(Dierolf-Schwanengel \cite{DIS79}). More general cases of minimal (so-called \emph{admissible}) semidirect products were studied by
Remus and Stoyanov \cite{RES91}.
By \cite{MEG98},
$\R^n \leftthreetimes \R_+$ is minimal for every $n \in \N$.

\item Every connected semisimple Lie group with finite center, e.g., $SL_n(\R)$, $n \geq 2$ (Remus and Stoyanov \cite{RES91}).

\item The full unitary group $U(H)$ (Stoyanov \cite{STO84}).
\een
\end{remark}


One of the immediate difficulties
is the fact that 
 minimality is not preserved by quotients and
(closed) subgroups. See for example item (5) with minimal $\R
\leftthreetimes \R_+$ where its canonical quotient $\R_+$ (the
positive reals) and the closed normal subgroup $\R$ are nonminimal.
As a contrast note that in a minimal \emph{abelian} group every
closed subgroup is minimal \cite{DPS89}.

In 1983 Pestov raised the conjecture that every topological group is
a group retract of a minimal group. Note that if $f: M \to G$ is a
group retraction then necessarily $G$ is a quotient of $M$ and also
a closed subgroup in $M$. Arhangel'ski\u\i \ asked the following
closely related questions:

 \begin{quest} \label{Arh}
 \emph{(\cite{Ar}, \cite{MILLR})}
 Is every topological group a quotient of a minimal
group? Is every topological group a closed subgroup of a minimal
group?
\end{quest}

By a result of Uspenskij \cite{USP98} every topological group is a
subgroup of a minimal group $M$ which is Raikov-complete,
topologically simple and Roelcke-precompact.

 Recently a positive answer to
 Pestov's conjecture (and hence to 
 Question \ref{Arh} of
 Arhangel'ski\u\i) was obtained in \cite{MEG08}.
 The proof is based on
methods (from \cite {MEG95}) of constructing minimal groups using
group representations on Banach spaces and involving generalized
Heisenberg groups.

According to \cite{MEG95} every locally compact \emph{abelian} group
is a group retract of a minimal locally compact group.
 It is
an open question whether the same is true in the nonabelian case.

\begin{quest} \label{q:loc}
\emph{(\cite{MEG95}, \cite{MEG08} and \cite{CHR})} Is it
true that every locally compact group $G$ is a group retract
 (at least a subgroup or a quotient) of a locally compact minimal group?
\end{quest}

 A more general natural question is the following:

 \begin{quest} \cite{MEG95} \label{q:gen} Let $\K$ be a
 certain class of topological groups and  $\underline{min}$ denotes
 the class of all minimal groups. Is it true that every $G\in \K$ is
 a group retract of a group $M \in \K \cap  \underline{min}$ ?
\end{quest}

So Corollary \ref{c:main}  
gives a partial answer to Questions \ref{q:loc} and
\ref{q:gen} in the class $\K:=\NA$ of 
non-archimedean groups.
\begin{remark}
Note that by \cite[Theorem 7.2]{MEG08} we can present any
topological group $G$ as a group retraction $M \rightarrow G,$ where
$M$ is a minimal group having the same weight and character as $G.$
Furthermore, if $G$ is Raikov-complete then $M$ also has the same property.
 These results provide in particular a positive answer to Question \ref{q:gen} in the following basic classes: second countable groups,
metrizable groups, Polish groups.
\end{remark}

\subsection{Minimality properties of actions}

\begin{defi} \label{d:t-ex}
Let $\alpha:G\times X\rightarrow X, \ \alpha(g,x)=gx$ be a
continuous action of a Hausdorff topological group $(G,\sigma)$ on a
Hausdorff topological space $(X,\tau)$. The action $\alpha$ is said
to be: \ben
 \item \emph{algebraically exact} if 
 $ker_{\alpha}:=\{g \in G: \ gx=x \ \ \forall x \in X\}$ is the trivial subgroup $\{e\}$.
 \item \emph{topologically exact} (\emph{$t$-exact}, in short) if there is no strictly
coarser, not necessarily Hausdorff, group topology
$\sigma'\varsubsetneq \sigma$ on $G$ such that $\alpha$ is
$(\sigma', \tau, \tau)$-continuous. \een
\end{defi}

 \begin{remark}
 \ben
\item
Every topologically exact action is algebraically exact.
Indeed, otherwise $ker_{\alpha}$ is a non-trivial subgroup
in 
 $G.$ Then the preimage group topology $\sigma' \subset \sigma$ on
$G$ induced by the onto homomorphism $G \to G/ker_{\alpha}$ is not
Hausdorff (in particular, it differs $\sigma$) and the action
remains $(\sigma', \tau, \tau)$-continuous.
\item
On the other hand, if $\alpha$ is algebraically exact then it is
topologically exact if and only if for every strictly coarser
\emph{Hausdorff} group topology $\sigma'\varsubsetneq \sigma$ on $G$
the action $\alpha$ is not $(\sigma', \tau, \tau)$-continuous.
Indeed, since $\alpha$ is algebraically exact and
$(X,\tau)$ is Hausdorff then every coarser group topology $\sigma'$
on $G$ which makes the action $(\sigma', \tau, \tau)$-continuous
must be Hausdorff.
 \een
\end{remark}

Let $X$ be a locally compact group and $Aut(X)$ be the group of all
automorphisms endowed with the \emph{Birkhoff topology}
 (see \cite[\S 26]{HR} and \cite[p. 260]{DPS89}).
 Some authors use the name \emph{Braconnier topology} (see \cite{CM}).

The latter is a group topology on $Aut(X)$ and has a local base
 formed by the sets
$$\mathcal{B}(K,O):=\{f \in Aut(X): \ f(x) \in Ox \ \text{and} \ f^{-1}(x) \in Ox \ \ \forall \ x \in K\}$$
where $K$ runs over compact subsets and $O$ runs over neighborhoods
of the identity in 
 $X.$ In the sequel $Aut(X)$ is always equipped
with the Birkhoff topology.
It equals to the Arens \emph{$g$-topology} 
\cite{AR46, CM}. 
 If $X$ is compact then the Birkhoff
topology coincides with the usual compact-open topology. If $X$ is
discrete then the Birkhoff topology on $Aut(X) \subset X^X$
coincides with the pointwise
topology. 

\begin{lemma} \label{are2} In each of the following cases the action of $G$ on $X$ is t-exact:
 \ben
 \item \cite{MEG95}  Let $X$ be a locally compact group and $G$ be a subgroup of $Aut(X)$.
  \item  Let $G$ be a topological subgroup of $\Homeo(X)$, the group of all autohomeomorphisms
  of a compact space $X$ with the compact open topology.
 \item
 Let $G$ be a subgroup of $Is(X,d)$ the group of all isometries of a metric space $(X,d)$ with the pointwise topology.

 \een
 \end{lemma}
 \begin{proof}
 Straightforward.
 \end{proof}

\subsection{From minimal dualities to minimal
groups}\label{sub:dual}

In this 
 subsection we recall some definitions and results from \cite{MEG95,
MEG08}.

Let $E,F,A$ be abelian additive 
topological groups. A map
$w:E\times F \rightarrow A$ is said to be {\it biadditive }if the
induced mappings
$$w_{x}:F\rightarrow A,\ w_{f}:E\rightarrow A,
\ w_{x}(f):=w(x,f)=:w_{f}(x)$$ are homomorphisms for all $x\in E$
and $f\in F$.

 A biadditive mapping
$w:E\times F \rightarrow A$ is {\it separated} if for
every pair $(x_{0},f_0)$ of nonzero elements there exists a pair
$(x,f)$ such that $f(x_0)\neq 0_{A}$ and  $f_0(x)\neq 0_{A}$.

 A continuous separated biadditive mapping $w:(E,\sigma)\times (F,\tau)
\rightarrow A$ is {\it minimal} if for every coarser
pair $(\sigma_1, \tau_1)$ of Hausdorff group topologies
$\sigma_1\subseteq \sigma,\ \tau_1\subseteq \tau$ such that $w:
(E,\sigma_1)\times (F,\tau_1) \rightarrow A$ is continuous, it
follows that $\sigma_1=\sigma$ and $\tau_1=\tau.$

Let $w:E \times F \rightarrow A$ be a continuous biadditive mapping.
 Consider the action:
 $w^{\triangledown}:F\times(A \oplus E)\to A \oplus E, \quad w^{\triangledown}(f,(a,x))=(a+w(x,f),x).$
Denote by
$H(w)=(A \oplus E)\leftthreetimes F$ the topological semidirect
product of $F$ and the direct sum $A \oplus E$. The group operation
on $H(w)$ is defined as follows: for a pair
$$u_{1}=(a_{1},x_{1},f_{1}), \ u_{2}=(a_{2},x_{2},f_{2}) $$ we
define $$u_{1}u_{2}=(a_{1}
+a_{2}+f_{1}(x_{2}),x_{1}+x_{2},f_{1}+f_{2})$$ where,
 $f_{1}(x_{2})=w(x_{2},f_{1}).$ Then $H(w)$ becomes a Hausdorff
topological group which is said to be a {\it generalized Heisenberg
group} (induced
by $w$). 

Let $G$ be a topological group and let $w:E\times F\rightarrow A$ be
a continuous biadditive mapping. A continuous {\it birepresentation} of $G$ in
$w$ is a pair $(\alpha_1,\alpha_2)$ of continuous actions by group 
automorphisms $\alpha_1:G\times E\rightarrow E$ and
$\alpha_2:G\times F\rightarrow F$ such that $w$ is $G$-invariant,
i.e., $w(gx,gf)=w(x,f)$.

The birepresentation $\psi$ is said to be \emph{t-exact} if
$ker(\alpha_1) \cap ker(\alpha_2)=\{e\}$ and for every strictly
coarser \emph{Hausdorff} group topology  on $G$ the birepresentation
does not remain continuous.
For instance, if one of the actions $\alpha_1$ or $\alpha_2$ is
t-exact then clearly $\psi$ is t-exact.

Let $\psi$ be a continuous $G$-birepresentation $$\psi=(w:E\times
F\rightarrow A, \ \alpha_1:G\times E\rightarrow E, \ \alpha_2:G\times
F\rightarrow F).$$ The topological semidirect product
$M(\psi):=H(w)\leftthreetimes_\pi G$ is said to be the \emph{induced
group}, where the action $\pi:G\times H(w)\rightarrow H(w) $ is
defined by $$\pi(g,(a,x,f))=(a,gx,gf).$$ 

\begin{fact} \label{fac:bir}
Let $w:E\times F\rightarrow A$ be a minimal biadditive mapping and
$A$ is a minimal group. Then

\ben
\item \cite[Corollary 5.2]{DIM04}  The Heisenberg group $H(w)$ is minimal.
\item
\emph{(See \cite[Theorem 4.3]{MEG95} and \cite{MEG08})} If $\psi$ is a $t$-exact $G$-birepresentation in $w$
 then the induced group $ M(\psi)$ is minimal.
 \een
\end{fact}

\begin{fact} \label{lca} \cite{MEG95}
Let $G$ be a locally compact abelian group and $G^*:=Hom(G,\T)$ be
the dual (locally compact) group. Then the canonical evaluation
mapping
$$
G \times G^* \to \T
$$
is minimal and the corresponding Heisenberg group $H=(\T \oplus G)
\leftthreetimes G^*$ is minimal.
\end{fact}



\section{Some facts about non-archimedean groups and uniformities}

\subsection{Non-archimedean uniformities}


For information on {\it uniform spaces}, we refer the reader to \cite{Eng} (in terms of \emph{entourages}) and to
\cite{isbo} (via \emph{coverings}). If $\mu$ is a uniformity for $X$ in terms of coverings, then the collection
of elements of $\mu$ which are {\it finite} coverings of $X$ forms a
base for a topologically compatible uniformity for $X$ which we denote by $\mu_{fin}$ (the precompact replica of $\mu$).


A {\it partition} of a set $X$ is a covering of $X$
consisting of pairwise disjoint subsets of $X$. Due to 
Monna (see \cite[p.38]{ro} for more details), a uniform space
$(X,\mu)$ is {\it non-archimedean} if it has a base consisting of
partitions of $X$. In terms of entourages, it is equivalent to
saying that there exists a base $\mathfrak{B}$ of the uniform
structure such that every entourage $P \in \mathfrak{B}$ is an
equivalence relation. Equivalently, iff its \emph{large uniform
dimension} (in the sense of Isbell \cite[p. 78]{isbo}) is zero.

A metric space $(X,d)$ is said to be an \emph{ultra-metric space}
(or, \emph{isosceles} \cite{Lem03})
if $d$ is an
\emph{ultra-metric}, i.e., it satisfies {\it the strong triangle
inequality}
$$d(x,z) \leq max \{d(x,y), d(y,z)\}.$$
The definition of \emph{ultra-semimetric} is the same as ultra-metric apart from the fact  that
the condition $d(x,y)=0$ need not imply $x=y$.
For every ultra-semimetric $d$ on $X$ every $\varepsilon$-covering 
 $\{B(x, \varepsilon): \ x \in X\}$
by the open balls is a clopen partition of $X$.

Furthermore, a uniformity is non-archimedean iff it is generated by
a system $\{d_i\}_{i \in I}$ of {\it ultra-semimetrics}.
The following result (up to obvious reformulations) is well known.
See, for example, \cite{isbo} and \cite{ispa}.

\begin{lemma}
\label{unifp}
 Let $(X,\mu)$ be a non-archimedean uniform space. Then both
 $(X,\mu_{fin})$
 and the uniform completion $\hXm$ of $(X,\mu)$
 are non-archimedean uniform spaces.
\end{lemma}
%

\subsection{Non-archimedean groups}

The class $\mathcal{NA}$ of all non-archimedean groups is quite
large. Besides the results of this section see Theorem
\ref{t:condit} below.
The prodiscrete (in particular, the profinite) groups are in
 $\mathcal{NA}.$
All $\NA$ groups are totally disconnected and for every locally compact
totally disconnected group $G$ both $G$ and $Aut(G)$ are $\mathcal{NA}$ (see Theorems 7.7 and 26.8 in \cite{HR}).
Every abelian $\mathcal{NA}$ group is embedded into a product of discrete groups.

  The minimal groups $(\Z, \tau_p), S_X, \Homeo(\{0,1\}^{\aleph_0})$ (in 
  items (1),
(2) and (3) of Remark \ref{r:min}) are 
non-archimedean. By Theorem \ref{t:1} the Heisenberg group
$H_X=(\Bbb{Z}_2 \oplus V) \leftthreetimes V^{\ast}$ is $\NA$ for
every Stone space $X$. It is well known
 that there exist $2^{\aleph_0}$-many nonhomeomorphic metrizable Stone spaces.

Recall that 
every topological group can be identified with  a subgroup of $\Homeo(X)$ for some compact 
$X$ and also with a subgroup of $Is(M,d)$, topological group of
isometries of some metric space $(M,d)$ endowed with the pointwise
topology, \cite{Te}.
Similar characterizations are true for
$\mathcal{NA}$ with compact zero-dimensional spaces $X$ and
ultra-metric spaces $(M,d)$. See Lemma \ref{l:firstpart} and Theorem
\ref{t:condit} below.

We 
will use later the following simple observations. Let $X$ be a
Stone space (compact 
zero-dimensional space) and $G$ be a
topological subgroup of $\Homeo(X)$.
For every finite clopen partition $P=\{A_1,\ldots, A_n\}$ of $X$ 
  define the subgroup $$M(P):=\{g \in G: \ \
g A_k=A_k \ \forall \ 1\leq k\leq n\}.$$ Then all subgroups of this
form defines a local base (subbase, if we consider only two-element partitions $P$)
of the original compact-open topology on $G \subset \Homeo(X)$.
 So for every Stone
space  $X$ the topological group $\Homeo(X)$ is non-archimedean.
More generally, for every non-archimedean uniform space $(X,\mu)$
consider the group $Unif(X,\mu)$ of all uniform automorphisms of $X$
(that is, the bijective functions $f: X \to X$ such that both $f$
and $f^{-1}$ are $\mu$-uniform). Then $Unif(X,\mu)$ is a
non-archimedean topological group
in the topology induced by the uniformity of uniform convergence.

%


\begin{lemma} \label{l:firstpart}
The following assertions are equivalent: \ben
\item $G$ is a non-archimedean topological group.
\item The right (left) uniformity on $G$ is non-archimedean.
\item $\dim \beta_G G =0$, where $\beta_G G$ is the maximal $G$-compactification \cite{MES01} of $G$.
\item $G$ is a topological subgroup of $\Homeo(X)$ for some compact
zero-dimensional space $X$ (where $w(X)=w(G)$).
\item
$G$ is a topological subgroup of $Unif(Y,\mu)$ for some
non-archimedean uniformity $\mu$ on a set $Y$. \een
\end{lemma}
\begin{proof}
For the sake of completeness we give here a sketch of the proof. The equivalence of (1) and (3) was established by Pestov \cite[Prop.
3.4]{Pest-free}.
The
equivalence of (1), (2) and (3) is  \cite[Theorem
3.3]{MES01}.

(1) $\Rightarrow$ (2) Let $\{H_i\}_{i \in I}$ be a local base at
$e$ (the neutral element of $G$), where each $H_i$ is an open (hence, clopen) subgroup of 
$G.$
Then the corresponding decomposition of 
$G=\cup_{g \in G} H_ig$ by right $H_i$-cosets defines an equivalence relation
$\Omega_i$ and the set 
 $\{\Omega_i\}_{i \in I}$ is a
base of the right uniform structure $\mu_r$ on $G$.

(2) $\Rightarrow$ (3) If the right uniformity $\mu$ is
non-archimedean then by Lemma \ref{unifp} the completion
$(\widehat{X},\widehat{\mu_{fin}})$ of its precompact replica
(Samuel compactification of $(X,\mu)$) is again non-archimedean. Now
recall (see for example \cite{MES01}) that this completion is just
the greatest $G$-compactification $\beta_G G$ (the $G$-space analog
of the Stone-\v{C}ech compactification) of $G$.


(3) $\Rightarrow$ (4) A result in \cite{MEG89} 
implies that  there exists a zero-dimensional proper
$G$-compactification $X$ of the $G$-space $G$ (the left action of
$G$ on itself) with $w(X)=w(G).$  Then the natural homomorphism
$\varphi: G\rightarrow \Homeo(X)$ is a topological group embedding.

(4)$\Rightarrow$ (5)  Trivial because $\Homeo(X)=Unif(X,\mu)$ for compact $X$ and its unique compatible uniformity $\mu$.

(5) $\Rightarrow$ (1) The non-archimedean uniformity $\mu$ has a
base $\mathfrak{B}$ where each $P \in \mathfrak{B}$ is an
equivalence relation. Then the subsets
$$M(P):=\{g \in G: \ \ (gx,x) \in P \ \forall x \in X\}.$$
form a local base of $G$. Observe that $M(P)$ is a subgroup of $G$.
\end{proof}




 $\mathcal{NA}$-ness of a dense subgroup implies that of the whole group. Hence
the Raikov-completion of $\mathcal{NA}$ groups are again
$\mathcal{NA}$.
 Subgroups,
quotient groups and (arbitrary) products of $\mathcal{NA}$ groups
are also 
$\mathcal{NA}.$ Moreover the class $\mathcal{NA}$ is
closed under group extensions. 

\begin{fact} \label{l:Hig} \cite[Theorem 2.7]{Hig} If both $N$ and $G/N$ are 
  $\mathcal{NA},$ then so is $G$.
\end{fact}
For the readers convenience we reproduce here the proof from
\cite{Hig}.
\begin{proof}
Let $U$ be a neighborhood of
$e$ in G. We shall find an open subgroup $H$ contained in $U$. We
choose neighborhoods $U_0$, $V$ and $W$ of $e$ in $G$ as follows.
First let $U_0$ be such that $U_0^{2} \subseteq U.$ By the
assumption, there is an open subgroup $M$ of $N$ contained in $N
\cap U_0$. Let $V\subseteq U_0$ be open with $V=V^{-1} $ and $V ^{3}
\cap N\subseteq M.$ We denote by $\pi$ the natural homomorphism $G
\rightarrow G/N.$ Since $\pi(V)$ is open in $G/N$, it contains an
open subgroup $K.$ We set $W = V \cap \pi^{-1}(K).$ We show that
$W^{2} \subseteq WM.$ Suppose that $w_0,w_1 \in W.$ Since $\pi(w_0),
\pi(w_1) \in K$, we have $\pi(w_0w_1) \in K.$ So there is $w_2 \in
W$ with 
$\pi(w_2) = \pi(w_0w_1).$  Then $w_2^{-1}w_0w_1 \in N \cap W^{3}
\subseteq M$, and hence $w_0w_1 \in w_2M.$ Using this result and
also the fact that $M$ is a subgroup of $N$ we obtain by induction
that $W^{k}\subseteq WM \quad \forall k\in \Bbb{N}.$ Now let $H$ be
the subgroup of $G$ generated by $W$. Clearly,
$H=\bigcup_{k=1}^{\infty}W^{k}.$ Then $H$ is open and
$$H\subseteq WM \subseteq U_0^{2} \subseteq U$$ as desired.
\end{proof}

\begin{corol} \label{l:com}
Suppose that $G$ and $H$ are non-archimedean groups and that $H$ is
a $G$-group. Then the semidirect product $H\leftthreetimes G $ is
non-archimedean.
\end{corol}


\begin{example} \label{exa:cyc} (Dierolf  and
Schwanengel \cite{DIS79}) Every discrete group $H$ is a group
retract of a locally compact non-archimedean minimal group.

More precisely, let $\Bbb Z_2$ be the discrete cyclic group of order
$2$ and let $H$ be a discrete topological group. Let $G:=\Bbb
Z_2^{H}$ be endowed with the product topology. Then
$$\sigma: H\rightarrow Aut(G), \ \sigma(k)((x_{h})_{h\in
H}):=(x_{hk})_{h\in H} \  \forall k\in H, (x_{h})_{h\in H}\in G$$
is a homomorphism. The topological semidirect
(wreath) product
$G\leftthreetimes_{\sigma} H$ is a locally
compact non-archimedean minimal group having $H$ 
as a retraction.
\end{example}

Corollary \ref{c:main} below provides a generalization.

\section{The Heisenberg group associated to a Stone
space}\label{s:main}


Let $X$ be a Stone space. 
Let $V=(V(X), \triangle)$ be the discrete group of all clopen
subsets in $X$ with respect to the symmetric difference.
 As usual one may identify $V$ with the group $V:=C(X,\Bbb{Z}_2)$ of all
continuous functions $f:X\rightarrow \Bbb{Z}_2.$

Denote by $V^{\ast}:=hom(V,\T)$ the Pontryagin dual of $V.$
Since $V$ is a Boolean group 
every character $V \to \T$ can be identified with a homomorphism into the unique
2-element subgroup $\Omega_2=\{1, -1\}$, a copy of $\Z_2$. The same
is true for the characters on $V^*$, hence the natural evaluation
map 
$w:V \times V^*  \to  \T$ ($w(x,f)=f(x)$) can be restricted
naturally to $V \times V^* \to \Z_2$. Under this identification
$V^{\ast}:=hom(V,\Z_2)$ is a closed (hence compact) subgroup of the
compact group $\Bbb{Z}_2^{V}.$
Clearly, the groups $V$ and $\Bbb Z_2$, being discrete, are
non-archimedean. The group $ V^{\ast}=hom(V,\Bbb{Z}_2)$ is also
non-archimedean since it is a subgroup of $\Bbb{Z}_2^{V}.$

In the sequel $G$ is an
arbitrary non-archimedean group. $X$ is  its associated  Stone
space,  that is,  $G$ is a topological subgroup of $\Homeo(X)$ (see
Lemma \ref{l:firstpart}). $V$ and $V^{\ast}$ are the non-archimedean
groups associated to the Stone space $X$ we have mentioned at
the beginning of this subsection. We intend to show using the
technique introduced in Subsection \ref{sub:dual}, among others,
that $G$ is a topological group retract of a non-archimedean minimal
group.

\begin{thm} \label{t:1}
For every Stone space $X$ the (locally compact 2-step nilpotent)
Heisenberg group $H=(\Bbb{Z}_2 \oplus V)\leftthreetimes V^{\ast}$ is
minimal and non-archimedean.
\end{thm}
\begin{proof}
Using Fact \ref{lca} (or, by direct arguments) it is easy to see
that
the continuous separated biadditive mapping 
 $$w: V \times V^* \to  \Z_2$$ is minimal. Then by Fact \ref{fac:bir}.1 the
corresponding Heisenberg group 
$H$ is minimal. $H$ is 
non-archimedean by Corollary \ref{l:com}.
\end{proof}

\begin{lemma} \label{l:wet}
Let $G$ be a topological subgroup of
 $\Homeo(X)$ 
 for some Stone space $X$ (see Lemma \ref{l:firstpart}).
Then $w(G) \leq w(X)=w(V)=|V|=w(V^{\ast}).$
\end{lemma}
\begin{proof}
Use the facts that in our setting $V$ is discrete and $V^*$ is compact.
Recall also that (see e.g.,  \cite[Thm. 3.4.16]{Eng})
$$w (C(A,B)) \leq w(A) \cdot w(B)$$
for every locally compact Hausdorff space $A$ 
(where the space $C(A,B)$ is endowed with the compact-open
topology).
%
%
\end{proof}

The action of $G \subset \Homeo(X)$ on $X$ and the
functoriality of the Stone duality induce
the actions on $V$ and $V^*$. 
 More precisely, we have 
$$\alpha: G \times V\rightarrow V, \ \ \alpha(g,A)=g(A)$$
and
$$\beta: G\times V^{\ast} \rightarrow V^{\ast},\ \ \ \beta(g,f):=gf,
\ \ (gf)(A)=f(g^{-1}(A)).$$
 Every translation under these actions is a continuous group
automorphism. Therefore we have the associated group homomorphisms:
$$i_{\alpha}: G \to Aut(V) $$
$$i_{\beta}: G \to Aut(V^*)$$
The pair $(\alpha, \beta)$ is a birepresentation of $G$ on $w: V
\times V^* \to \Z_2.$ Indeed,
$$w(gf,g(A))=(gf)(g(A))=f(g^{-1}(g(A)))=f(A)=w(f,A).$$

\begin{lemma} \label{l:emb}
\ben
\item
Let $G$ be a topological subgroup of $\Homeo(X)$ for some Stone space 
$X.$
The action $\alpha: G \times V\rightarrow V$ induces a topological
group embedding $i_{\alpha}: G \hookrightarrow Aut(V)$.
\item The natural evaluation 
map
$$\delta: X \to V^*, \ x \mapsto \delta_x, \ \ \ \delta_x(f)=f(x)$$
is a topological $G$-embedding.
\item
The action $\beta: G\times V^{\ast} \rightarrow V^{\ast}$ induces a
topological group embedding $i_{\beta}:~G \hookrightarrow Aut(V^*)$.
\item The pair $\psi:=(\alpha, \beta)$ is a t-exact birepresentation of $G$ on $w: V \times V^* \to \Z_2.$
\een
\end{lemma}
\begin{proof}
(1)
Since $V$ is discrete, the Birkhoff topology on $Aut(V)$ coincides
with the pointwise topology. Recall that the topology on $G$
inherited from $\Homeo(X)$ is defined by the local subbase
$$H_A:=\{g \in G: \ \ g A=A\}$$
where $A$ runs over nonempty clopen subsets in $X$. Each $H_A$ is a
clopen subgroup of $G$. On the other hand the pointwise topology on
$i_{\alpha}(G) \subset Aut(V)$ is generated by the local subbase of
the form
$$\{i_{\alpha} (g) \in i_{\alpha}(G): \ \ g A=A\},$$
So, $i_{\alpha}$ is a topological group embedding.


(2) Straightforward.

(3)
Since $V^*$ is compact, the Birkhoff topology on $Aut(V^*)$
coincides with the compact open topology.

The action of $G$ on $X$ is t-exact. Hence, by (2) it follows that
the action $\beta$ cannot be continuous under any weaker group
topology on $G$. Now it suffices to show that the action $\beta$ is
continuous.

The topology on $V^* \subset \Z_2^V$ is a pointwise topology
inherited from $\Z_2^V$. So it is enough to show that for every
finite family $A_1, A_2, \cdots, A_m$ of nonempty clopen subsets in
$X$ there exists a neighborhood $O$ of $e \in G$ such that
$(g\psi)(A_k)=\psi(A_k) $ for every $k$. Since
$(g\psi)(A_k)=\psi(g^{-1}(A_k))$
 we may define $O$ as
$$O:=\cap_{k=1}^m H_{A_k}$$

(Another way to prove (3) is to combine (1) and \cite[Theorem 26.9]{HR}).

(4) $\psi=(\alpha, \beta)$ is a birepresentation as we already
noticed before this lemma. The t-exactness is a direct consequence
of (1) or (3) together with 
Fact \ref{are2}.1.
\end{proof}


\begin{thm} \label{t:2}
The topological group $$M:=M(\psi)= H(w)\leftthreetimes_\pi G=
((\Bbb{Z}_2 \oplus V)\leftthreetimes V^{\ast})\leftthreetimes_\pi
G$$
is a 
non-archimedean minimal group.
\end{thm}
\begin{proof}
By  Corollary \ref{l:com}, $M$ is non-archimedean. Use Theorem
\ref{t:1}, Lemma \ref{l:emb} and Fact \ref{fac:bir} to conclude that
$M$ is a minimal group.
\end{proof}

\begin{corol} \label{c:main}
Every (locally compact) non-archimedean group $G$ is a group
retract of a (resp., locally compact) 
minimal non-archimedean group $M$ where $w(G)=w(M)$.
\end{corol}
\begin{proof}
Apply Theorem \ref{t:2} taking into account Fact \ref{are2}.1 
and the local compactness of the groups $\Bbb{Z}_2, V, V^{\ast}$
(resp., $G$).
\end{proof}


\begin{remark}
Another proof of Corollary \ref{c:main} can be obtained by the
following way. By Lemma \ref{l:emb} a non-archimedean group $G$ can
be treated as a subgroup of the group of all automorphisms
$Aut(V^*)$ of the compact abelian group $V^*$. In particular, the
action of $G$ on $V^*$ is t-exact. The group $V^*$ being compact is
minimal. Since $V^*$ is abelian one may apply \cite[Cor. 2.8]{MEG95}
which implies that $V^* \leftthreetimes G$ is a minimal topological
group.
By Lemmas \ref{l:firstpart} and \ref{l:wet} we may assume that  $w(G)=w(V^* \leftthreetimes G)$.
\end{remark}


\section{More characterizations of non-archimedean groups}

The results and discussions above lead to the following
list of characterizations (compare Lemma \ref{l:firstpart}).

\begin{thm} \label{t:condit}
The following assertions are equivalent: \ben
\item $G$ is a non-archimedean topological group.

\item $G$ is a topological subgroup of the automorphisms group (with the pointwise topology) $Aut(V)$ for some
discrete 
Boolean ring $V$ (where $|V|=w(G)$).
\item
$G$ is embedded into the symmetric topological group $S_{\kappa}$
(where $\kappa=w(G)$).
\item
$G$ is a topological subgroup of the group $Is(X,d)$ of all
isometries of an 
ultra-metric space $(X,d)$, with the topology of pointwise
convergence.
\item The right (left) uniformity on $G$ can be generated by a
system of right (left) invariant 
ultra-semimetrics.

\item $G$ is a topological subgroup of the automorphism group $Aut(K)$ for some compact abelian group $K$ (with $w(K)=w(G)$).
\een
\end{thm}
\begin{proof}
(1) $\Rightarrow$ (2) As in Lemma \ref{l:emb}.1.

(2) $\Rightarrow$ (3) Simply take the embedding of $G$ into $S_V
\cong S_{\kappa}$, with $\kappa=|V|=w(G)$.

(3) $\Rightarrow$ (4) Consider the two-valued ultra-metric on the discrete
space
$X$ with $|X|=\kappa$.

(4) $\Rightarrow$ (5) For every $z \in X$ consider the left
invariant ultra-semimetric
$$
\rho_z(s,t):=d(sz,tz).
$$
Then the collection $\{\rho_z\}_{z \in X}$ generates the left uniformity of $G$.

(5) $\Rightarrow$ (1) Observe that for every right invariant
ultra-semimetric $\rho$ on $G$ and $n \in \N$ the set
$$
H:=\{g \in G: \ \rho(g,e) < 1/n \}
$$
is an open subgroup of $G$.

(3) $\Rightarrow$ (6) Consider the natural (permutation of
coordinates) action of $S_{\kappa}$ on the
usual Cantor additive group $\Z_2^{\kappa}$.
 It is easy to see that this action implies the natural embedding of $S_{\kappa}$
 (and hence, of its subgroup $G$) into the group $Aut(\Z_2^{\kappa})$.


(6) $\Rightarrow$ (1)
Let $K$ be a
compact abelian group and $K^{\ast}$ be its (discrete) dual.
 By \cite[Theorem 26.9]{HR}
the natural map $\nu: g \mapsto \tilde{g}$ defines a topological
anti-isomorphism of $Aut(K)$ onto $Aut(K^{\ast}).$ Now,  $K^{\ast}$
is discrete, hence, $Aut(K^{\ast})$ is  non-archimedean as a
subgroup of the symmetric group $S_{K^{\ast}}.$ Since $G$ is a
topological subgroup of $Aut(K)$ we conclude that $G$ is also
non-archimedean (because its opposite group $\nu(G)$ being a
subgroup of $Aut(K^{\ast})$ is non-archimedean).
\end{proof}

\begin{remark}
\ben
\item
Note that the universality of $S_\N$ among Polish groups was
proved by Becker and Kechris (see \cite[Theorem 1.5.1]{bk}). The
universality of $S_{\kappa}$ for $\mathcal{NA}$ groups with weight $\leq
\kappa$ can be proved similarly. It appears in the work of
Higasikawa, \cite[Theorem 3.1]{Hig}.
\item Isometry groups of ultra-metric spaces studied among others by Lemin and Smirnov \cite{Les86}.
Note for instance that \cite[Theorem 3]{Les86} implies the
equivalence (1) $\Leftrightarrow$ (4). Lemin \cite{Lem84}
established that a metrizable group is non-archimedean iff it has a
left invariant compatible ultra-metric.
\item In item (6) of Theorem \ref{t:condit} it is essential that the compact group $K$ is \emph{abelian}.
For every connected non-abelian compact group $K$ the group $Aut(K)$ is not $\NA$ containing a nontrivial continuous image of $K$.
\item Every non-archimedean group admits a topologically faithful unitary representation on a Hilbert space.
It is straightforward for $S_X$ (hence, also for its subgroups) via permutation of
coordinates linear action.
\een
\end{remark}

\section{Automorphizable actions and epimorphisms in topological groups}
\label{s:aut}

Resolving a longstanding principal problem by K. Hofmann,
Uspenskij \cite{Us-epic} has shown
that in the category of Hausdorff topological groups epimorphisms need not have a dense range.
Dikranjan and Tholen
present
in \cite{DT} a rather direct proof of this important result of Uspenskij.
Pestov gave later a criterion
\cite{Pest-epic, pest-wh} (Fact \ref{pest}) which we will use below in Theorem \ref{epic=dense}. This criterion
is closely related to the natural concept of the \emph{free topological $G$-group} $F_G(X)$ of a $G$-space $X$
introduced by the first author \cite{Me-F}.
It is a natural $G$-space version of the usual \emph{free topological group}.
A topological (uniform)
$G$-space $X$ is said to be \emph{automorphizable} if $X$ is a topological (uniform)
$G$-subspace of a $G$-group $Y$ (with its right uniform structure).
Equivalently, if the universal morphism $X \to F_G(X)$ of $X$ into
the free topological (uniform) $G$-group $F_G(X)$ of the (uniform)
$G$-space $X$ is an embedding.

\begin{fact} \label{pest} \emph{(Pestov \cite{Pest-epic, pest-wh})}
Let $f: M \to G$ be a continuous homomorphism between Hausdorff topological groups.
Denote by $X:=G/H$ the left coset $G$-space, where
$H$ 
is the closure of the subgroup $f(M)$ in 
$G.$ The following are equivalent: \ben
\item 
$f: M \to G$ is an epimorphism.
\item 
The free topological $G$-group $F_G(X)$ of the $G$-space $X$ is trivial.
 \een
\end{fact}
%

Triviality in (2) means, `as trivial as possible',
isomorphic to the cyclic discrete group.

   Let $X$ be the 
   $n$-dimensional cube
$[0,1]^n$ or the $n$-dimensional sphere $\mathbb{S}_n$. Then by
\cite{Me-F} the free topological $G$-group $F_G(X)$ of the $G$-space
$X$ is trivial for every $n \in \N$, where $G=\Homeo(X)$ is the
corresponding homeomorphism group. So, one of the possible examples
of an epimorphism which is not dense can be constructed as the
natural embedding $H \hookrightarrow G$
where $G=\Homeo(\mathbb{S}_1)$ and $H=G_z$ is the stabilizer 
of a point $z \in \mathbb{S}_1$.
The same example serve as an original counterexample in the paper of Uspenskij \cite{Us-epic}.

In contrast, for Stone spaces, we have:

\begin{prop} \label{p:aut}
Every continuous action of a topological group $G$ on a Stone space
$X$ is automorphizable (in $\mathcal{NA}$).
Hence the canonical $G$-map $X \to F_G(X)$ 
is an embedding.
\end{prop}
\begin{proof}
Use item (2) of Lemma \ref{l:emb}.
\end{proof}

Roughly speaking
this result says that
the action by conjugations of a subgroup $H$ of a
non-archimedean group $G$ on $G$ 
reflects all possible difficulties of the Stone actions.
Below in Theorem \ref{t:AE} we extend Proposition \ref{p:aut} to a
much larger class of actions on 
non-archimedean uniform spaces, where $X$ need not be compact. This
will be used in Theorem \ref{epic=dense} about epimorphisms into $\NA$-groups.

\begin{defi} \label{d:quasib} \cite{me-fr}
Let $\pi: G \times X \to X$ be an action and $\mu$ be a uniformity on $X$.
We say that the action is \emph{$\pi$-uniform} if
for every $\varepsilon \in \mu$ and $g_0 \in G$ there exist: $\delta \in \mu$
and a neighborhood 
$O$ of $g_0$ in $G$ such that
$$(gx,gy) \in \varepsilon  \ \ \ \ \forall \ (x,y) \in \delta, \ g \in O.$$
\end{defi}

It is an easy observation that if the action $\pi: G \times X \to X$ is $\pi$-uniform
and all orbit maps $\tilde{x}: G \to X$ are continuous then $\pi$ is continuous.

\begin{lemma} \label{lem:puni} \cite{me-fr}
Let 
 $\mu$ be a uniformity on a  $G$-space $X$ which generates its topology. Then the action $\pi: G \times X \to X$ is $\pi$-uniform
in each of the following cases: \ben
\item $X$ is a $G$-group and $\mu$ is the right or left uniformity on $X$.
\item $X$ is the coset $G$-space $G/H$ with respect to the standard right uniformity 
(which is always compatible with the topology).
\item $\mu$ is the uniformity of a $G$-invariant metric.
\item $X$ is a compact $G$-space and $\mu$ is the unique compatible uniformity on $X$.
\een
\end{lemma}

%

A function $||\cdot ||: G \to [0,\infty)$ on an abelian group $(G,+)$ is an \emph{ultra-norm} \cite{War}
if $||u||=0 \Leftrightarrow u=0$,  $||u||=||-u||$ and
$$||u+v|| \leq \max \{||u|| , ||v|| \} \ \ \ \forall \ u,v \in G.$$
A group $(G,+,||\cdot ||)$ with an ultra-norm $||\cdot ||$ is an
\emph{ultra-normed space}. The definition of an ultra-seminorm is
understood. It is easy to see that if the topology on $(G,+)$ can be
generated by a system of ultra-seminorms then $G$ is a
non-archimedean group (cf. Theorem \ref{t:condit}, the equivalence
(1) $\Leftrightarrow$ (5))
and its right (=left) uniformity is just the uniform structure induced on $G$ by the given system of ultra-seminorms.
Every abelian 
 non-archimedean metrizable group admits an ultra-norm (see Theorems
6.4 and 6.6 in \cite{War}).

\subsection{Arens-Eells linearization theorem for actions}

Recall that the well known Arens-Eells linearization theorem (cf.
\cite{AE})
asserts that every uniform (metric) space can be
(isometrically) embedded into a locally convex vector space (resp.,
normed space).
For a metric space $(X,d)$ one can define a real normed space $(A(X), ||\cdot||)$
as the set of all formal linear combinations
$$
\sum_{i=1}^n c_i(x_i-y_i)
$$
where $x_i,y_i \in X$ and $c_i \in \R$. For every $u \in A(X)$ one may define the norm by
$$
||u||:=\inf \{\sum_{i=1}^n |c_i| d(x_i,y_i) :  \ \ u=\sum_{i=1}^n c_i(x_i-y_i) \}.
$$

Now if $(X,z)$ is a pointed space with some $z \in X$ then $x \mapsto x - z$ defines an isometric embedding
of $(X,d)$ into $A(X)$ (as a closed subset).

This theorem on isometric linearization of metric
spaces can be naturally extended to the case of non-expansive
semigroup actions provided that the metric is bounded \cite{Me-cs},
or, assuming only that the orbits are bounded \cite{Schroder}.
Furthermore, suppose that an action of a \emph{group} $G$ on a
metric space $(X,d)$ with bounded orbits is only uniform in the
sense of Definition \ref{d:quasib} (and not necessarily
non-expansive). Then again such an action admits an isometric
$G$-linearization on a normed space.

Here we give a non-archimedean version of Arens-Eells type theorem 
for uniform group actions.

\begin{thm} \label{t:AE}
Let $\pi: G \times X \to X$ be a continuous $\pi$-uniform action of
a topological group $G$ on a non-archimedean
Hausdorff
uniform space $(X,\mu)$. \ben
\item
Then there
exist a $\mathcal{NA}$
Hausdorff
Boolean $G$-group $E$ and a
uniform
$G$-embedding
$$
\a: X \hookrightarrow E
$$
such that $\a(X)$ is closed. Hence, $(X,\mu)$ is uniformly $G$-automorphizable (in $\mathcal{NA}$).

\item
Let $(X,d)$ be an 
ultra-metric space and 
suppose there exists a $d$-bounded orbit $Gx_0$ for some 
$x_0 \in X.$ Then there exists an ultra-normed Boolean $G$-group $E$
and an isometric $G$-embedding $\a: X \hookrightarrow E$ such that
$\a(X)$ is closed.
\item Every 
 ultra-metric space is isometric to a closed subset of an ultra-normed Boolean group.
 \een
\end{thm}
\begin{proof}
(1) Every non-archimedean uniformity $\mu$ on $X$ can be generated
by a system $\{d_j\}_{j \in J}$ of ultra-semimetrics. Furthermore
one may assume that $d_j \leq 1.$
Indeed, every uniform partition of $X$ leads to the naturally defined
0, 1 ultra-semimetric. 
We can suppose in addition that $X$ contains a $G$-fixed
point $\theta$. Indeed, adjoining if necessary a fixed point $\theta$ and defining
$d_j(x,\theta)=d_j(\theta ,x)=1$ for every $x \in X$, we get again an ultra-semimetric.

Furthermore one may assume that for any finite collection $d_{j_1}, d_{j_2}, \cdots, d_{j_m}$ from the system $\{d_j\}_{j \in J}$
the 
ultra-semimetric $max\{d_{j_1}, d_{j_2}, \cdots, d_{j_m}\}$ also
belongs to our system.
%

Consider the \emph{free Boolean group} $(P_F(X), +)$ over the set 
$X.$ The elements of $P_F(X)$ are finite subsets of $X$ and the
group operation $+$ is the
symmetric difference of subsets. The zero element 
(represented by the empty subset) we denote by $\textbf{0}$.
Clearly, $u=-u$ for every $u \in P_F(X)$.

For every nonzero $u=\{x_1,x_2,x_3,
\cdots, x_m\} \in P_F(X)$, define the support $supp(u)$ as $u$ treating it as a subset
of $X$.
So $x \in X$ is a \emph{support element} of $u$ iff 
$x \in \{x_1,x_2,x_3, \cdots, x_m\}.$ Let us say that $u$ is
\emph{even} (\emph{odd}) if the number of support elements $m$ is
even (resp., odd). Define the natural homomorphism $sgn: P_F(X) \to
\Z_2=\{\overline{0},\overline{1}\}$, where, $sgn(u)=\overline{0}$
iff $u$ is even. We denote by $E$ the subgroup
$sgn^{-1}(\overline{0})$ of all even elements in $P_F(X)$.

Consider the natural set embedding $$\iota: X \hookrightarrow P_F(X), \  \iota(x)=\{x\}.$$
Sometimes we will identify $x \in X$ and $\iota(x)=\{x\} \in P_F(X).$

Define also another embedding of sets
$$\a: X \to E, \ \ \a(x) =  x - \theta.$$ 
Observe that $\a(x)-\a(y)= \iota(x) -\iota(y) = x-y$ for every $x,y \in X$.

By a \emph{configuration} we mean a finite subset of $X \times X$
(finite relations). Denote by $Conf$ the set of all configurations.
We can think 
of any $\omega \in Conf$ as a finite set of some pairs
$$
\omega=\{(x_1,x_2), (x_3, x_4), \cdots , (x_{2n-1}, x_{2n})\},
$$
where all $\{x_i\}_{i=1}^{2n}$ are (not necessarily distinct) elements of $X$. If 
$x_i \neq x_k$ for all distinct $1\leq i,k \leq 2n$ then
$\omega$ is said to be \emph{normal}. 
For every $\omega \in Conf$ the sum
$$
u:=\sum_{i=1}^{2n} x_i=\sum_{i=1}^{n} (x_{2i-1}-x_{2i}).
$$
necessarily belongs to $E$ and we say that $\omega$ \emph{represents} $u$ or, that $\omega$ is an $u$-\emph{configuration}.
Notation $\omega \in Conf(u)$. We denote by $Norm(u)$ the set of all normal configurations of $u$.
If $\omega \in Norm(u)$ then necessarily $\omega \subseteq supp(u) \times supp(u)$
and $supp(u)= \{x_1,x_2,  \cdots , x_{2n}\}$.
So, $Norm(u)$ is a finite set for any given $u \in E$.

Let $j \in J$. Our aim is to define an ultra-seminorm $||\cdot||_j$ on the Boolean group
$(E, +)$ such that $d_j(x,y)=||x - y||_j$. 
For every configuration $\omega$ we define its $d_j$-\emph{length} by
$$\varphi_j(\omega)=\max_{1 \leq i \leq n} d_j(x_{2i-1},x_{2i}).
$$


\noindent \textbf{Claim 1:}
For every even nonzero element $u \in E$ and every $u$-configuration
$$
\omega=\{(x_1,x_2), (x_3, x_4), \cdots , (x_{2n-1}, x_{2n})\}
$$
define the following elementary reductions:
\ben
\item
Deleting a trivial pair $(t,t)$. That is, deleting the pair $(x_{2i-1}, x_{2i})$ whenever $x_{2i-1}=x_{2i}$.
\item Define the \emph{trivial inversion at $i$} of $\omega$ as the replacement of $(x_{2i-1}, x_{2i})$
by the pair in the reverse order $(x_{2i}, x_{2i-1})$.
\item Define the \emph{basic chain reduction rule} as follows.
Assume that there exist distinct $i$ and $k$ such that $x_{2i}= x_{2k-1}.$
We delete in the configuration $\omega$ two pairs $(x_{2i-1}, x_{2i})$, $(x_{2k-1}, x_{2k})$
and add the following new pair $(x_{2i-1}, x_{2k})$.
\een

Then in all three cases we get again an $u$-configuration.
The reductions (1) and (2) do not change
the $d_j$-length of the configuration.
Reduction (3) cannot exceed the $d_j$-length.

\begin{proof}  Comes directly 
from the axioms of ultra-semimetric. In the proof of (3) observe that
$$x_{2i-1} + x_{2i} + x_{2k-1} + x_{2k}=x_{2i-1} +  x_{2k}$$
in $E$.
This ensures that the new configuration is again an $u$-configuration.
\end{proof}

\noindent \textbf{Claim 2:}
For every even nonzero element $u \in E$ and every $u$-configuration $\omega$
 there exists a normal $u$-configuration $\nu$ such that $\varphi_j(\nu) \leq \varphi_j(\omega)$.
\begin{proof}
Using Claim 1 after finitely many reductions of $\omega$ we get a normal $u$-configuration $\nu$
such that $\varphi_j(\nu) \leq \varphi_j(\omega)$.
\end{proof}

Now we define the desired ultra-seminorm $||\cdot||_j$.
For every 
$u \in E$ define
$$||u||_j=\inf_{\omega \in Conf(u)} \varphi_j(\omega).$$

\noindent \textbf{Claim 3:} For every nonzero $u \in E$ we have $$||u||_j=\min_{\omega \in Normal(u)} \varphi_j(\omega).$$
\begin{proof}
By Claim 2 
 it is enough to compute $||u||_j$  via normal $u$-configurations. So, since $Norm(u)$ is finite, we may replace $\inf$ by $\min$.
\end{proof}

\noindent \textbf{Claim 4:} $||\cdot||_j$ is an ultra-seminorm on $E$.
\begin{proof}
Clearly,  $||u||_j \geq 0$ and $||u||_j=||-u||_j$
(even $u=-u$) for every $u \in E$.
For the $\textbf{0}$-configuration  $\{(\theta,\theta)\}$ we obtain that 
$||\textbf{0}||_j \leq d_j(\theta,\theta)=0$. So $||\textbf{0}||_j=0$. 
We have to show that
$$||u+v||_j \leq \max \{||u||_j , ||v||_j \} \ \ \ \forall \ u,v \in E.$$
Assuming the contrary, there exist  configurations
$$\{(x_i,y_i)\}_{i=1}^n, \ \ \   \{(t_i,s_i)\}_{i=1}^m$$
with
$$u=\sum_{i=1}^n (x_i-y_i), \ v=\sum_{i=1}^m (t_i-s_i)$$ such that
$$||u+v||_j> c:=\max\{\max_{1 \leq i \leq n} d_j(x_i,y_i),\max_{1 \leq i
\leq m} d_j(t_i,s_i)\}$$ but this contradicts the definition of
$||u+v||_j$ because $$u+v=\sum_{i=1}^n (x_i-y_i)+\sum_{i=1}^m
(t_i-s_i)$$ and hence

$$\omega:=\{(x_1,y_1), \cdots, (x_n,y_n), (t_1,s_1), \cdots ,(t_m,s_m)\}$$
is a configuration of $u+v$ with $||u+v||_j > \varphi_j(\omega)=c$, a  contradiction to the definition of $||\cdot||_j$.
\end{proof}

\textbf{Claim 5:} $\alpha: (X,d_j) \hookrightarrow (E,||\cdot||_j), \a(x)=x-\theta$ is an isometric embedding, that
is,
$$||x-y||_j=d_j(x,y) \ \ \ \forall \ x,y \in X.$$
\begin{proof}
By Claim 3 we may compute via normal configurations.
For the element $u=x-y \neq \textbf{0}$ only possible \emph{normal } configurations are $\{(x,y)\}$ or $\{(y,x)\}$. So $||x-y||_j=d_j(x,y).$
\end{proof}

\vskip 0.5cm

\textbf{Claim 6:} For any given $u \in E$ with $u \neq \textbf{0}$ we have
$$||u||_{j} \geq \min\{d_j(x_i,x_k): \ \  x_i,x_k \in supp(u), \ x_i \neq x_k\}.$$
\begin{proof} Easily comes from Claims 2 and 3.
\end{proof}

\textbf{Claim 7:} For any given $u \in E$ with $u \neq \textbf{0}$ there exists $j_0 \in J$ such that $||u||_{j_0} >0$.
\begin{proof}
Since $u \neq \textbf{0}$ we have at least two elements in 
$supp(u).$  Since $(X,\mu)$ is Hausdorff the system $\{d_j\}_{j \in
J}$ of ultra-semimetrics separates points of $X$. So some finite
subsystem $d_{j_1}, d_{j_2}, \cdots, d_{j_m}$ separates points of
$supp(u)$. By our assumption
the 
ultra-semimetric $d_{j_0}:=max\{d_{j_1}, d_{j_2}, \cdots, d_{j_m}\}$
belongs to our system $\{d_j\}_{j \in J}$. Then
$$\min\{d_{j_0}(x_i,x_k): \ \  x_i,x_k \in supp(u), x_i \neq x_k\} >0.$$
Claim 6 implies that $||u||_{j_0} >0$.
\end{proof}

It is easy to see that the family $\{||\cdot||_j\}_{j \in J}$ of
ultra-seminorms
induces a non-archimedean 
group topology on the Boolean group $E$ and
a 
 non-archimedean uniformity $\mu_*$ which is the right (=left) uniformity on $E$.
By Claim 7 the topology on $E$ is Hausdorff.

We have the natural group action $$\overline{\pi}: G \times E \to E, (g,u) \mapsto gu$$
induced by the given action $G \times X \to X$.
Clearly, $g(u+v)=gu +gv$ for every $(g,u,v) \in G \times E \times E$.
So this action is by automorphisms.
Since $g \theta =\theta$ for every $g \in G$ it follows that $\a: X \to E$ is a $G$-embedding.

Now we show that the action $\overline{\pi}$ of $G$ on $E$ is uniform
and continuous. Indeed, the original action on $(X,\mu)$ is $\pi$-uniform. Hence, for
every $j \in J$, $\varepsilon> 0$ and $g_0 \in G$, there exist: 
 a finite subset $\{j_1,\ldots,  j_n\}$ of $J, \ \delta >0$ and a neighborhood $O(g_0)$ of $g_0$ in $G$ such that
$$d_j(gx,gy) \leq \varepsilon  \ \ \ \ \forall \ \max_{1 \leq i \leq
n} d_{j_i}(x,y)\leq \delta, \ g \in O.$$
Then by Claim 3 it is easy to see that
$$||gu||_j \leq \varepsilon  \ \ \ \ \forall \ \max_{1 \leq i \leq
n}||u||_{j_i} < \delta, \ g \in O.$$ This implies that the action
$\overline{\pi}$ of $G$ on $(E, \mu_*)$ is uniform.
Claim 5 implies that $\a: X \hookrightarrow E$ is a topological $G$-embedding.
Since
$\a(X)$
algebraically spans $E$ it easily follows that every
orbit mapping 
 $G \to E, \ g \mapsto gu$ is continuous for every
$u \in E$. So we can conclude that $\overline{\pi}$ is continuous (see 
the remark after Definition \ref{d:quasib}) and $E$ is a $G$-group.

\vskip 0.3cm


Finally we check that $\a(X)$ is closed in $E$. Let $u \in E$
and $u \notin \a(X)$. 
Since $u-x+\theta \neq \textbf{0}$  for every $x \in X$,
we can suppose that 
there are at least two elements in $supp(u) \cap (X \setminus
\{\theta\})$. Similarly to the proof of Claim 7 we may choose $j_0 \in J$ and $\varepsilon_1 >0$ such that
$$\varepsilon_1:=\min\{d_{j_0}(x_i,x_k): \ \  x_i,x_k \in supp(u), \ x_i \neq x_k\} >0.$$
Furthermore, one may assume in addition that $$\varepsilon_2:=\min\{d_{j_0}(x_i,\theta): \ \  x_i \in supp(u), \ x_i \neq \theta\} >0.$$
Define $\varepsilon_0:=min\{\varepsilon_1,\varepsilon_2\}$.

For every $x \in X,$ every normal configuration $\omega$ of
$u-x+\theta \neq \textbf{0}$ contains an element $(s,t)$ such that
$\{s,t\}\subset supp(u)\cup
\{\theta\}.$ Therefore,
 $$\varphi_{j_0}(\omega)\geq d_{j_0}(s,t)\geq \varepsilon_0.$$  So by
Claim 3 we obtain $||u-x+\theta||_j \geq \varepsilon_0$ for every $x
\in X$.

%
%

Summing up we finish the proof of (1).

\vskip 0.5cm

(2) The proof in the second case is similar.
We only explain why we may suppose that $X$ contains a $G$-fixed point.
Indeed, as in the paper of Schr\"oder \cite[Remark 5]{Schroder}
we can look at $(X,d)$ as embedded into the space $exp(X)$ of all bounded closed subsets endowed with the standard Hausdorff metric $d_H$ 
defined by
$$
d_H(A,B):=\max \{\sup_{a \in A} d(a,B), \ \sup_{b \in B} d(A,b)\}.
$$ 
The closure $cl(Gx_0)$ of the orbit $Gx_0$ in $X$ is bounded and defines an element $\theta \in exp(X)$.
Consider the metric subspace $X':=X \cup \{\theta\} \subset exp(X)$. It is easy to see that
 the induced action of $G$ on $X'$ is well defined and remains uniform (Definition \ref{d:quasib}) with respect to the metric $d_H|_{X'}$.
 Clearly, $\theta$ is a $G$-fixed point in $X'.$ This implies that all orbit maps $G \to X'$ are continuous. It follows that the action of $G$ on
$X'$ is continuous (see the remark after Definition \ref{d:quasib}). 

Finally observe that since $d$ is an ultra-metric the Hausdorff metric $d_H$ on $exp(X)$  
is also an ultra-metric. Hence, 
$d_H|_{X'}$ is an ultra-metric on $X'$. 
To prove the strong triangle inequality for $d_H$ we will use the
following lemma. 

\begin{lemma}\label{lem:sti}
Let $(X,d)$ be an ultra-metric space and  $A,B,C$  subsets of $X.$
Then $$\sup_{a\in A} d(a,C) \leq \max \{\sup_{a\in
A} d(a,B), \ \sup_{b\in B} d(b,C)\}$$
\end{lemma}
\begin{proof}
Let $M:=\sup_{a\in A} d(a,C).$ Assuming the contrary,
$$M>d(a,B) \quad \forall a\in A$$ and also $$M>d(b,C) \quad \forall
b\in B.$$ Set $a_0\in A.$ Since $M>d(a,B) \quad \forall a\in A$, we
have in particular $M>d(a_0,B)$. So there exists $b_0\in B$ such that
$M>d(a_0,b_0).$ Now, $M>d(b,C) \quad \forall b\in B,$ hence, there
exists $c_0\in C$ such that $M>d(b_0,c_0).$ Since $d$ is an
ultra-metric we obtain that $M>d(a_0,c_0)\geq d(a_0,C).$ Since $a_0$
is an arbitrary element of $A$  we get that $M>\sup_{a\in A} d(a,C)=M.$  
This  clearly  contradicts  our assumption.
\end{proof}
We can now prove the strong triangle inequality for $d_H$. Using
Lemma \ref{lem:sti} twice we obtain that
$$\sup_{a\in A} d(a,C) \leq \max \{\sup_{a\in
A} d(a,B), \ \sup_{b\in B} d(b,C)\}$$ 
and also (by switching $A\leftrightarrow C$)  
$$\sup_{c\in C} d(A,c) \leq \max\{\sup_{b\in B} d(A,b), \ \sup_{c\in C} d(B,c)\}.$$ This
implies that $$d_H(A,C)\leq \max\{d_H(A,B), \ d_H(B,C)\}.$$

\vskip 0.7cm

(3) Directly follows from (2). 

\end{proof}

\begin{thm} \label{epic=dense}
Let $G$ be a non-archimedean group. If a continuous homomorphism $f: M \to G$ is
an epimorphism in the category of Hausdorff topological groups then $f(M)$ is dense in $G$.
\end{thm}
\begin{proof} Denote by $H$ the closure of the subgroup $f(M)$ in $G$. We have to show that $H=G$. Assuming the
contrary consider the \emph{nontrivial} Hausdorff coset $G$-space $G/H$.
Recall that the sets $$\tilde{U}:=\{(aH,bH): bH\subseteq UaH\},$$ where $U$ runs over
the neighborhoods of $e$ in $G$, form a uniformity base  on $G/H$. This
 uniformity (called the right uniformity) is compatible with the
quotient topology (see for instance \cite{DR81}).

The fact that $G$ is $\mathcal{NA}$ implies that the right
uniformity on $G/H$ is non-archimedean. Indeed, if $\mathcal{B}$ is
a local base at $e$ consisting of clopen subgroups then
$\tilde{\mathcal{B}}:=\{\tilde{U}: U\in \mathcal{B}\}$ is  a base
for the right uniformity of $G/H$ and its elements are equivalence
relations. To see this just use the fact that $H$ as well as the
elements of $\mathcal{B}$ are all subgroups of $G.$
 By Lemma \ref{lem:puni}.2 the natural
left action $\pi:G\times G/H\to G/H $ is $\pi$-uniform. Obviously
this action is  also continuous. Hence, we can apply
Theorem
\ref{t:AE}.1 to conclude that the nontrivial $G$-space $X:=G/H$ is
$G$-automorphizable in $\mathcal{NA}$. In particular,
we obtain that there exists a \emph{nontrivial} equivariant morphism of the $G$-space $X$ to 
a Hausdorff $G$-group $E$. This implies that the free topological $G$-group $F_G(X)$ of the $G$-space $X$ is not trivial.
Now by the criterion of Pestov (Fact \ref{pest}) we conclude that $f: M \to G$ is not an epimorphism.
\end{proof}


\bibliographystyle{amsalpha}

\end{document}